\documentclass{amsart}

\usepackage{amsmath, amssymb, amsthm}
\usepackage{hyperref}
\usepackage{cleveref}
\usepackage{thmtools}
\usepackage{mathtools}
\usepackage{xcolor}
\usepackage[new]{old-arrows}
\usepackage{tikz-cd}

\newtheorem{theorem}{Theorem}[section]
\newtheorem{glem}[theorem]{Ghost Lemma}

\newtheorem{proposition}[theorem]{Proposition}

\newtheorem{lemma}[theorem]{Lemma}

\newtheorem*{theorem*}{Theorem}
\newtheorem*{corollary*}{Corollary}
\newtheorem*{lemma*}{Lemma}
\newtheorem*{proposition*}{Proposition}

\theoremstyle{definition}

\newtheorem{remark}[theorem]{Remark}
\newtheorem{example}[theorem]{Example}

\newtheorem{construction}[theorem]{Construction}
\newtheorem{chunk}[theorem]{}
\newtheorem*{chunk*}{}

\newcommand{\cat}[2][R]{\mathsf{#2}(#1)}
\newcommand{\class}[1]{\mathsf{#1}}
\newcommand{\catd}[1][R]{\mathsf{D}(#1)}
\newcommand{\catdbf}[1][R]{\mathsf{D_b^f}(#1)}
\newcommand{\level}[3][R]{\operatorname{level}_{#1}^{#2}#3}  
\newcommand{\mfp}{\mathfrak{p}}
\newcommand{\mfm}{\mathfrak{m}}
\newcommand{\ext}[4][R]{\operatorname{Ext}_{#1}^{#2}(#3,#4)}
\newcommand{\gdim}[2][R]{\operatorname{Gdim}_{#1}#2}
\newcommand{\GPdim}[2][R]{\operatorname{Gpd}_{#1}#2}
\newcommand{\GFdim}[2][R]{\operatorname{Gfd}_{#1}#2}
\newcommand{\pdim}[2][R]{\operatorname{pd}_{#1}#2}
\newcommand{\fdim}[2][R]{\operatorname{fd}_{#1}#2}
\newcommand{\idim}[2][R]{\operatorname{id}_{#1}#2}
\newcommand{\h}[2][]{\operatorname{H}_{#1}(#2)}
\newcommand{\cy}[2][]{\operatorname{Z}_{#1}(#2)}
\newcommand{\co}[2][]{\operatorname{C}_{#1}(#2)}
\newcommand{\bo}[2][]{\operatorname{B}_{#1}(#2)}
\newcommand{\Hom}[3][R]{\operatorname{Hom}_{#1}(#2,#3)}

\newcommand{\erk}{\operatorname{E}_{R}(k)}
\newcommand{\im}{\operatorname{im}}
\newcommand{\coker}{\operatorname{coker}}
\newcommand{\cone}{\operatorname{cone}}
\newcommand{\depth}{\operatorname{depth}}

\newcommand{\Gid}{\operatorname{Gid}_R}

\newcommand{\catdb}[1][R]{\mathsf{D_b}(#1)}

\title[G-Levels of Perfect Complexes and a Bass formula for levels]{G-Levels of Perfect Complexes\\ and a Bass formula for levels}

\author[L.W.\,Christensen]{Lars Winther Christensen}
\address{Texas Tech University, TX 79409, U.S.A.}
\email{lars.w.christensen@ttu.edu}
\urladdr{https://www.math.ttu.edu/~lchriste}

\author[A.\,Kekkou]{Antonia Kekkou}
\address{University of Utah, UT 84112, U.S.A.}
\email{kekkou@math.utah.edu}
\urladdr{https://sites.google.com/view/antonia-kekkou/}

\author[J.\,Lyle]{Justin Lyle}
\address{Auburn University, AL 36849, U.S.A.}
\email{jll0107@auburn.edu}
\urladdr{https://jlyle42.github.io/justinlyle}

\author[Z.\,Nason]{Zachary Nason}
\address{University of Nebraska, NE 68588, U.S.A.}
\email{znason2@huskers.unl.edu}
\urladdr{https://zach-nason.github.io/}

\author[A.J.\,Soto Levins]{Andrew J. Soto Levins}
\address{Texas Tech University, TX 79409, U.S.A.}
\email{ansotole@ttu.edu}
\urladdr{https://sites.google.com/view/andrewjsotolevins}

\date{\today}



\thanks{L.W.C.\ was partly supported by Simons Foundation collaboration grant 962956}
\thanks{A.K. was partly supported by National Science Foundation grant DMS-200985}
\begin{document}

\begin{abstract}
  We prove that a commutative noetherian ring $R$ is Gorenstein of
  dimension at most $d$ if $d+1$ is an upper bound on the G-levels of
  perfect $R$-complexes. For $R$ local, we prove a formula for levels,
  with respect to injective or Gorenstein injective $R$-modules, of
  $R$-complexes with finitely generated homology; it mimics Bass'
  classic formula for injective dimension of finitely generated
  $R$-modules.
\end{abstract}

\maketitle

\section{Introduction}

\noindent
Throughout this paper, $R$ denotes a commutative noetherian ring. Let
$\class{C}$ be a collection of objects in $\catdb$, the bounded
derived category over $R$. The number of mapping cones needed, up to
summands, finite direct sums, and shifts, to build an $R$-complex $M$
from $\class{C}$ is known as the level of $M$ with respect to
$\class{C}$ and denoted by $\level{\class{C}}{M}$. That is, levels
stratify the derived category of $R$ using its triangulated
structure. Levels with respect to the class $\{ R \}$, sensibly known
as $R$-levels, have been studied extensively. From the first study of
$R$-levels by Avramov, Buchweitz, Iyengar, and Miller
\cite{Avramov/Buchweitz/Iyengar/Miller:2010} it is known that $R$ is
regular of Krull dimension $d$ if and only if every complex in
$\catdbf$ has $R$-level at most $d + 1$.

In this paper, we are concerned with levels with respect to the
collection of finitely generated Gorenstein projective modules,
denoted $\cat{G}$, and the collections of Gorenstein injective and
Gorenstein flat modules, denoted $\cat{GI}$ and $\cat{GF}$. A first
major result in this area was obtained by Aihara, Araya, Iyama,
Takahashi, and Yoshiwaki \cite{AAITY-14}, who proved that if $R$ is
Gorenstein of finite Krull dimension, then the $\cat{G}$-levels of
complexes in $\catdbf$ are bounded above by $\max\{2, \dim R + 1\}$.
A converse was established by Awadalla and Marley
\cite{Awadalla/Marley:2022}, who proved that $R$ is Gorenstein of
finite Krull dimension if there is an upper bound on the
$\cat{G}$-levels of complexes in $\catdbf$.

Our first main result, Theorem \ref{cmgperfuacimpliesgorenstein},
improves the result of Awadalla and Marley, as we show that $R$ is
Gorenstein of finite Krull dimension if there is an upper bound on the
$\cat{G}$-levels of perfect $R$-complexes. This parallels a well-known
characterization of regular rings; see for example Krause
\cite{Krause:2024}. In \Cref{boundForOneDimensionalRings} we combine
Theorem \ref{cmgperfuacimpliesgorenstein} with results from
\cite{AAITY-14} to give a comprehensive characterization of Gorenstein
rings in terms of $\cat{G}$-levels.

Our second main result, Theorem \ref{GIInequality}, establishes an
upper bound for the $\cat{GI}$-levels of complexes in $\catdb$.  For
$R$ local of positive depth, we obtain in \Cref{GIBass} a formula for
the $\cat{GI}$-level of complexes in $\catdbf$ whose homology modules
have finite Gorenstein injective dimension. This formula mimics Bass's
classic formula for injective dimension of finitely generated
$R$-modules.

We close by examining $\cat{GF} $-levels of complexes in $\catdb$. It
follows from \cite{AAITY-14} that the $\cat{GF}$-level of a complex
$M$ in $\catdb$ is at most $\max\{2, \GFdim[R]{\h{M}}\}$, where
$\h{M}$ is taken to be a module. We show that ``$\max\{2,$" is needed
by constructing a complex of $\cat{GF}$-level $2$ whose homology
modules are Gorenstein flat. As the underlying ring $R$ is regular,
this is also a complex with flat homology modules whose level with
respect to the class of flat $R$-modules is, nevertheless, $2$.

\section{G-Levels of perfect complexes}

\noindent
We start by introducing the most central notation; for any unexplained
notation and terminology we refer the reader to Christensen, Foxby,
and Holm \cite{Christensen/Foxby/Holm:2024}.

For an $R$-complex
\[M\coloneqq \cdots \longrightarrow M_{i+1}
  \xrightarrow{\partial^M_{i+1}} M_i \xrightarrow{\partial^M_i}
  M_{i-1} \longrightarrow \cdots\] and every integer $i$ set
$\cy[i]{M}\coloneqq \ker(\partial^M_i)$,
$\bo[i]{M}\coloneqq\im(\partial^M_{i+1})$,
$\co[i]{M}\coloneqq \coker(\partial^M_{i+1})$, and
$\h[i]{M}\coloneqq \cy[i]{M}/\bo[i]{M}$.

\begin{chunk}
  Let $\class{C}$ be a collection of objects in the derived category
  $\catd$. Recall from \cite{Avramov/Buchweitz/Iyengar/Miller:2010}
  that the $\class{C}$-level of a complex $M$ in $\catd$, denoted
  $\level{\class{C}}{M}$, is defined as:
  \begin{enumerate}
  \item[$(1)$] $\level{\class{C}}{M}=0$ if $M$ is $0$ in $\catd$.
  \item[$(2)$] $\level{\class{C}}{M}=1$ if $M$ is nonzero, but can be
    built from objects in $\class{C}$ using shifts, summands, and
    finite direct sums.
  \item[$(3)$] For $n>1$, $\level{\class{C}}{M}=n$ if $n$ is the
    infimum of $i$ such that there is a triangle
    \[K\longrightarrow L\oplus M\longrightarrow N\longrightarrow\]
    with $\level{\class{C}}{K}=1$ and $\level{\class{C}}{N}=i-1$.
  \end{enumerate}
\end{chunk}

\begin{chunk}
  Let $\class{C}$ be a collection of objects in $\catd$. A morphism
  $\alpha\colon M\rightarrow N$ in $\catd$ is $\class{C}$-ghost if the
  induced maps
  \[\ext{n}{C}{M}\longrightarrow\ext{n}{C}{N}\]
  are zero for all integers $n$ and all objects $C$ in $\class{C}$.
\end{chunk}

The next result was first proved by Kelly \cite[Theorem
3]{Kelly:1965}.

\begin{glem}
  \label{ghostlemma}
  Let $\class{C}$ be a collection of objects from $\catd$ and
  $\alpha_{i}\colon M_{i}\rightarrow M_{i+1}$ for $0\leq i\leq n-1$ a
  sequence of morphisms in $\catd$ such that
  $\alpha_{n-1}\alpha_{n-2}\cdots \alpha_{0}$ is a nonzero morphism in
  $\catd$. If each $\alpha_{i}$ is $\class{C}$-ghost, then one has
  $\level{\class{C}}{M_{0}}\geq n+1$.
\end{glem}

Let $\cat{G}$ denote the collection of finitely generated Gorenstein
projective $R$-modules.

\begin{lemma}
  \label{vanishinglemmawithbound}
  Let $M$ be a finitely generated $R$-module. If there exists an
  integer $b \ge1$ such that $\ext{n}{G}{M} = 0$ holds for all
  $n \ge b$ and all $G$ in $\cat{G}$, then $\ext{n}{G}{M} = 0$ holds
  for all $n \ge 1$ and all $G$ in $\cat{G}$.
\end{lemma}

\begin{proof}
  Assume that $\ext{n}{G}{M} = 0$ holds for all $n \ge b$ and all $G$
  in $\cat{G}$, and assume towards a contradiction that one has
  $\ext{n}{G}{M} \ne 0$ for some $n \ge 1$ and some module $G$ from
  $\cat{G}$. As $G$ is a $b^{\textnormal{th}}$ syzygy of a module $G'$
  in $\cat{G}$, one has
  \begin{equation*}
    \ext{b+n}{G'}{M} \cong \ext{n}{G}{M} \ne 0\,,
  \end{equation*}
  a contradiction.
\end{proof}

Ideas for the proof of the next result come from \cite[Theorem
3.3]{Awadalla/Marley:2022} and \cite[Proposition A.1.2]{Krause:2024}.

\begin{theorem}
  \label{improvedcmgperfuacimpliesgorenstein}
  Let $R$ be a commutative noetherian local ring. If there exists an
  integer $\ell$ such that $\level{\cat{G}}{P}\leq \ell$ holds for
  every perfect $R$-complex $P$, then $R$ is Gorenstein.
\end{theorem}

\begin{proof}
  Without loss of generality one can assume that $\ell$ is at least
  $2$. Let $k$ be the residue field of $R$, let $K$ be the Koszul
  complex on a minimal set of generators for the maximal ideal of $R$,
  and set $D \coloneqq \Hom{K}{\erk}$. Then $D$ is a bounded complex
  of injective $R$-modules and has finite length homology. To show
  that $R$ is Gorenstein, it is by \cite[Corollary
  19.5.12]{Christensen/Foxby/Holm:2024} enough to show that $D$ has
  finite projective dimension over $R$.

  To this end, let $F\xrightarrow{\simeq} D$ be a semi-free resolution
  over $R$ with $F$ degreewise finitely generated. For
  $n \ge \sup\h{D}$, one has
  $\pdim[R]{D} \le \pdim[R]{\co[n]{F}} + n$, so it suffices to fix an
  integer $s \ge \sup\h{D}$ and show that $\co[s]{F}$ has finite
  projective dimension; after a shift one can take $s = 0$.  As the
  injective dimension of $D$ over $R$ is finite, it follows that for
  every $R$-module $G$ one has $\ext{m}{G}{D} = 0$ for all
  $m>\idim{D}$. Thus, for every $n \ge 0$ and $G$ in $\cat{G}$, there
  is an integer $b_n$ such that
  $\ext{m}{G}{\co[n]{F}} = \ext{m+n}{G}{D} = 0$ for all $m \ge b_n$,
  see Christensen, Frankild, and Holm \cite[Corollary
  2.10]{Christensen/Frankild/Holm:2006}, whence one has
  $\ext{m}{G}{\co[n]{F}}=0$ for all $m \ge 1$ and all $G\in\cat{G}$ by
  Lemma \ref{vanishinglemmawithbound}.

  For $n \ge 0$, set
  \[X^{n} \coloneqq 0\longrightarrow
    F_{\ell+n}\xrightarrow{\partial_{\ell+n}}\cdots\xrightarrow{\partial_{n+1}}
    F_{n}\longrightarrow 0\,.\] Notice that we have an exact sequence
  of $R$-complexes
  \[
    0\longrightarrow \Sigma^{\ell+n}\co[\ell+n+1]{F} \longrightarrow
    X^{n}\longrightarrow \Sigma^{n}\co[n]{F} \longrightarrow 0\,,\]
  and from the associated exact sequence of Ext modules one gets for
  every $G\in\cat{G}$
  \begin{equation}
    \tag{$\ast$}
    \ext{m}{G}{X^{n}}=0 \quad\text{for}\quad m\neq -(\ell+n)  \text{ and }  m\neq -n\,.
  \end{equation}
  Now, consider the canonical morphisms
  $\alpha^{n}\colon X^{n}\rightarrow X^{n+1}$ for
  $n=0,\dots,\ell-1$. It follows from $(\ast)$ that they are all
  $\cat{G}$-ghost. Therefore, if the composite morphism
  $\alpha \coloneqq \alpha^{\ell-1}\alpha^{\ell-2}\cdots\alpha^{0}$ is
  nonzero in $\catdbf$, then Ghost Lemma \ref{ghostlemma} yields
  \[\level{\cat{G}}{X^{0}}\geq \ell+1\,,\]
  which contradicts the assumption. Thus, $\alpha$ is a zero morphism
  in $\catdbf$. Since $\alpha$ is a morphism of semi-projective
  $R$-complexes it is null-homotopic; in particular there exist
  homomorphisms
  \[\sigma_{\ell}\colon F_{\ell}\longrightarrow F_{\ell+1}
    \quad\text{and}\quad \sigma_{\ell-1}\colon
    F_{\ell-1}\longrightarrow F_{\ell}\] such that
  \[1^{F_{\ell}} =\alpha_{\ell} = \sigma_{\ell-1}\partial_{\ell} +
    \partial_{\ell+1}\sigma_\ell\,.\] 
  Consider the composite
  \[\bo[\ell]{F} \longhookrightarrow F_{\ell}
    \xrightarrow{\sigma_{\ell}}
    F_{\ell+1}\xrightarrow{\partial_{\ell+1}} \bo[\ell]{F}\,.\] For
  $x$ in $\bo[\ell]{F}$, one has
  \[x = 1^{F_\ell}(x) = \sigma_{\ell-1}\partial_{\ell}(x) +
    \partial_{\ell+1}\sigma_\ell(x) =
    \partial_{\ell+1}\sigma_\ell(x)\,,\] so
  $\partial_{\ell+1}\sigma_\ell(x)$ is a left inverse to the inclusion
  $\bo[\ell]{F} \hookrightarrow F_{\ell}$, which means that
  $\bo[\ell]{F}$ is a summand of a free module, whence
  $\pdim{\co[0]{F}}$ is at most $\ell+1$.
\end{proof}

\begin{remark}
  Notice that the complexes $X^n$ in the proof above have amplitude
  $\ell$. Thus, for $\ell \ge 2$ one gets the desired conclusion that
  $R$ is Gorenstein, as long as perfect $R$-complexes of amplitude
  $\ell$ have $\cat{G}$-level at most $\ell$.
\end{remark}

\begin{chunk}\label{auslanderbridger}
  Recall that if $R$ is local and $M \in \catdbf$ has finite
  Gorenstein dimension, then there is a classic formula due to
  Auslander and Bridger, see \cite[Theorem
  19.4.25]{Christensen/Foxby/Holm:2024}:
  \[\gdim[R]{M}=\depth R -\depth_R M\,.\]
\end{chunk}

The following fact will be used several times in the next couple of
proofs.

\begin{chunk}
  \label{letz}
  Let $\mfp$ be a prime ideal in $R$ and $P$ a perfect
  $R_\mfp$-complex. Since every finitely generated free
  $R_\mfp$-module is a localization of a finitely generated free
  $R$-module, Letz's proof of \cite[Lemma 3.9]{Letz:2021} shows that
  there exists a bounded complex $F$ of finitely generated free
  $R$-modules with $F_{\mfp}\simeq P$ in $\catdbf[R_\mfp]$.
\end{chunk}

\begin{lemma}
  \label{koszul-p}
  Let $R$ be Cohen--Macaulay of finite Krull dimension. There is a
  bounded complex $F$ of finitely generated free $R$-modules with
  $\level{\cat{G}}{F} \ge \dim R + 1$. Moreover, if $R$ is local then
  one can take $F$ to be the Koszul complex on a sequence of
  parameters for $R$, and in that case equality holds.
\end{lemma}

\begin{proof}
  Let $\mfm$ be a maximal ideal of $R$ with $\dim R_\mfm = \dim
  R$. Let $K$ be the Koszul complex on a sequence of parameters for
  $R_\mfm$ and choose by \ref{letz} a bounded complex $F$ of finitely
  generated free $R$-modules with $F_\mfm \simeq K$. In view of
  \cite[Lemma 2.4(6)]{Avramov/Buchweitz/Iyengar/Miller:2010} this
  explains the inequality in the display below. Per \cite[Corollary
  3.4]{Awadalla/Marley:2022} and \ref{auslanderbridger} the equalities
  hold as $K$ is the minimal free resolution of $\h[0]{K}$,
  \begin{align*}
    \level{\cat{G}}{F} & \ge \level[R_\mfm]{\cat[R_\mfm]{G}}{K}\\
                       & = \level[R_\mfm]{\cat[R_\mfm]{G}}{\h[0]{K}} \\
                       & = \gdim[R_\mfm]{\h[0]{K}} +1 \\ &= \dim R + 1\,.  \qedhere
  \end{align*}
\end{proof}

For $d = 0$ the conclusion in the next result is not optimal; see
Proposition \ref{lemmaforcmgperfuacimpliesgorenstein}.

\begin{theorem}
  \label{cmgperfuacimpliesgorenstein}
  Let $R$ be a commutative noetherian ring. If there exists an integer
  $d$ such that $\level{\cat{G}}{P}\leq d+1$ holds for every perfect
  $R$-complex $P$, then $R$ is Gorenstein of Krull dimension at most
  $d$.
\end{theorem}

\begin{proof}
  Let $\mfm$ be a maximal ideal of $R$; it suffices to show that the
  local ring $R_\mfm$ is Gorenstein of Krull dimension at most
  $d$. Let $P$ be a perfect $R_{\mfm}$-complex. Per \ref{letz} there
  exists a perfect $R$-complex $F$ with $F_{\mfm}\simeq P$ in
  $\catdbf[R_\mfm]$. Now, \cite[Lemma
  2.4(6)]{Avramov/Buchweitz/Iyengar/Miller:2010} yields the first
  inequality in the chain
  \[\level[R_{\mfm}]{\cat[R_{\mfm}]{G}}{P} =
    \level[R_{\mfm}]{\cat[R_{\mfm}]{G}}{F_{\mfm}} \leq
    \level{\cat{G}}{F} \leq d+1\,,\] which implies that $R_{\mfm}$ is
  Gorenstein by Theorem \ref{improvedcmgperfuacimpliesgorenstein}.

  Finally, one has $\dim R_{\mfm} \le d$ by Lemma \ref{koszul-p}.
\end{proof}

\begin{remark}
  Let $M$ be a complex in $\catdbf$ of finite Gorenstein projective
  dimension. It follows from \cite[Proposition
  9.1.27]{Christensen/Foxby/Holm:2024} that $M$ fits in a triangle
  with a perfect $R$-complex and a module from $\cat{G}$. This
  provides some measure of an \emph{a posteriori} explanation of why
  it suffices to bound the $\cat{G}$-level of perfect $R$-complexes to
  conclude that $R$ is Gorenstein.
\end{remark}

\section{G-levels over Gorenstein rings}

\begin{chunk}
  For an $R$-complex $M$, set
  $M^\oplus \coloneqq \bigoplus_{n \in \mathbb{Z}} M_n$. This notation
  comes in handy as $\level{\cat{G}}{M} = \level{\cat{G}}{M^\oplus}$
  holds for a complex $M$ with zero differential.
\end{chunk}

The next result follows from work of Aihara, Araya, Iyama, Takahashi,
and Yoshiwaki \cite[Theorem 4.1]{AAITY-14}; it can be also proved by
an argument dual to the proof of \Cref{GIInequality}.

\begin{proposition}
  \label{mainglevelthm}
  For every complex $M$ in $\catdbf$ one has
  \[\level{\cat{G}}{M}
    \le\max\{2,\gdim{\h{M}}^\oplus+1\}\,.\]
\end{proposition}

\begin{remark}
  \label{Gattain}
  The upper bound on $\cat{G}$-levels given above can be attained:
  Indeed, $\level{\cat{G}}{M} = \gdim{M} + 1$ holds for every
  $R$-module $M$ by \cite[Corollary~3.4]{{Awadalla/Marley:2022}},
  whence the bound is attained by every module of positive
  G-dimension. Thus, one may rewrite the upper bound as
  \[\level{\cat{G}}{M}
    \le \max\{2,\level{\cat{G}}{\h{M}^\oplus}\}\,.\] Moreover, there
  exists by \cite[Example 3.10]{Awadalla/Marley:2022} an $R$-complex
  $M$ with $\h{M}^\oplus$ in $\cat{G}$ but $\level{\cat{G}}{M} = 2$.
\end{remark}

The following is an immediate consequence of Proposition
\ref{mainglevelthm}:

\begin{theorem}
  \label{glevelbound}
  Let $R$ be Gorenstein and $M$ a complex in $\catdbf$. One has
  \[\level{\cat{G}}{M} \le \max\{2, \dim{R}+1\}\,.\]
\end{theorem}

Since $R$ is regular of finite Krull dimension if and only if there
exists a uniform upper bound on the $R$-level of all perfect
$R$-complexes, Theorem \ref{glevelbound} tells us that over a
nonregular Gorenstein ring there must exist a perfect complex $P$ such
that the inequality, $\level{\cat{G}}{P} \leq \level{R}{P}$, which
holds as $R$ belongs to $\cat{G}$, is strict.

\begin{proposition}
  \label{perfectcomlexesovernonregulargorensteinrings}
  Let $R$ be Gorenstein and not regular and $n\geq 2$ an
  integer. There exists a perfect $R$-complex $P$ with
  $\level{R}{P}=n+1$ and $\level{\cat{G}}{P}\leq 2$.
\end{proposition}

\begin{proof}
  We build off an argument by Altmann, Grifo, Monta\~{n}o, Sanders,
  and Vu \cite[Corollary
  2.3]{Altmann/Grifo/Montano/Sanders/Vu:2017}. Let $G$ be a module in
  $\cat{G}$ and not projective, i.e.\ of infinite projective
  dimension. If $F \to G$ is a free resolution, then the truncated
  complex
  \[P \coloneqq 0\longrightarrow
    F_{n}\longrightarrow\cdots\longrightarrow F_{0}\longrightarrow 0\]
  has $R$-level $n+1$ by the proof of \cite[Corollary
  2.3]{Altmann/Grifo/Montano/Sanders/Vu:2017}. Further, there is a
  triangle
  \[\Sigma^n\h[n]{P}\longrightarrow P\longrightarrow
    \h[0]{P}\longrightarrow\] in $\catdbf$, so
  $\level{\cat{G}}{P}\leq 2$ holds by \cite[Lemma
  2.4(2)]{Avramov/Buchweitz/Iyengar/Miller:2010} as both modules
  $\h[n]{P}$ and $\h[0]{P} \cong G$ are Gorenstein projective modules.
\end{proof}

The inequality in Theorem \ref{mainglevelthm} begs the question: What
happens if the $\cat{G}$-levels of complexes in $\catdbf$ are at most
$1$?

\begin{proposition}
  \label{lemmaforcmgperfuacimpliesgorenstein}
  If\, $\level{\cat{G}}{P}\leq 1$ holds for every perfect $R$-complex
  $P$, then $R$ is regular of Krull dimension $0$.
\end{proposition}

\begin{proof}
  It follows from \Cref{cmgperfuacimpliesgorenstein} that $R$ has
  Krull dimension $0$.  Let $\mfm$ be a maximal ideal in $R$; it
  suffices to show that $R_\mfm$ is regular. Let $P$ be a perfect
  complex over $R_{\mfm}$; per \ref{letz} there exists a perfect
  $R$-complex $F$ with $F_{\mfm}\simeq P$. Now, \cite[Lemma
  2.4(6)]{Avramov/Buchweitz/Iyengar/Miller:2010} gives the first
  inequality below
  \[\level[R_{\mfm}]{\cat[R_{\mfm}]{G}}{P} =
    \level[R_{\mfm}]{\cat[R_{\mfm}]{G}}{F_{\mfm}} \leq
    \level{\cat{G}}{F} \leq 1.\] One can now assume that $R$ is local.
  Let $K$ be the Koszul complex on a minimal generating set for the
  maximal ideal of $R$ and let $k$ be the residue field of $R$. Since
  $\level{\cat{G}}{K}=1$ holds, $K$ is isomorphic in the derived
  category to a complex $M$ that is a direct sum of shifts of modules
  from $\cat{G}$, and so we have
  \[K\simeq M\simeq \h[]{M}\simeq \h[]{K} =
    \bigoplus\Sigma^{i}\h[i]{K}\] Since $\h[0]{K}=k\neq 0$, we have
  $\h[i]{K}=0$ for all $i\neq 0$ by \cite[Proposition
  4.7]{Altmann/Grifo/Montano/Sanders/Vu:2017}, and so
  $\pdim{k}<\infty$. This implies $R$ is regular.
\end{proof}

\begin{theorem}
  \label{boundForOneDimensionalRings}
  Let $R$ be commutative noetherian; the following assertions hold.
  \begin{enumerate}\setlength{\itemsep}{1ex}
  \item[\textnormal{(a)}] The next conditions are equivalent.
    \begin{enumerate}\setlength{\itemsep}{.65ex}
    \item[$(i)$] $\sup\{\level{\cat{G}}{M} \mid M\in\catdbf \} = 2$
      holds.
    \item[$(ii)$]
      $\sup\{\level{\cat{G}}{P} \mid P \textnormal{ is a perfect
        $R$-complex} \} = 2$ holds.
    \item[$(iii)$] $R$ is Gorenstein with $\dim R =1$ or Gorenstein
      with $\dim R = 0$ and not regular.
    \end{enumerate}
    
  \item[\textnormal{(b)}] Let $d \ge 2$ be an integer. The following
    are equivalent.
    \begin{enumerate}\setlength{\itemsep}{.65ex}
    \item[$(i)$] $\sup\{\level{\cat{G}}{M} \mid M\in\catdbf \} = d+1$
      holds.
    \item[$(ii)$]
      $\sup\{\level{\cat{G}}{P} \mid P \textnormal{ is a perfect
        $R$-complex} \} = d+1$ holds.
    \item[$(iii)$] $R$ is Gorenstein with $\dim R = d$.
    \end{enumerate}
  \end{enumerate}
\end{theorem}

\begin{proof}
  (a): Assume that the equality in $(i)$ holds. It follows from
  Theorem \ref{cmgperfuacimpliesgorenstein} that $R$ is Gorenstein of
  Krull dimension at most $1$. If $\dim R$ equals $1$, then Lemma
  \ref{koszul-p} implies the existence of a perfect $R$-complex of
  $\cat{G}$-level $2$. If $\dim{R}=0$ holds, then $R$ is per
  \cite[Theorem 5.5]{Avramov/Buchweitz/Iyengar/Miller:2010} not
  regular, since there exists a complex in $\catdbf$ of
  $\cat{G}$-level $2$ and hence $R$-level at least $2$. Now,
  Proposition \ref{lemmaforcmgperfuacimpliesgorenstein} implies the
  existence of a perfect $R$-complex of $\cat{G}$-level $2$. Thus,
  $(i)$ implies both $(ii)$ and $(iii)$ and, more importantly, the
  arguments above apply verbatim to show that $(ii)$ implies
  $(iii)$. It remains to show that $(iii)$ implies~$(i)$:

  If $R$ is Gorenstein with $\dim R \le 1$, then $2$ is by
  \Cref{glevelbound} an upper bound on the $\cat{G}$-level of
  complexes in $\catdbf$. If $\dim R = 1$ holds, then Lemma
  \ref{koszul-p} implies the existence of a complex in $\catdbf$ of
  $\cat{G}$-level $2$. If $\dim R = 0$ holds and $R$ is not regular,
  then Proposition \ref{lemmaforcmgperfuacimpliesgorenstein} implies
  the existence of a complex in $\catdbf$ of $\cat{G}$-level $2$.

  (b): We proceed by induction on $d$. Let $d = 2$. If the equality in
  $(i)$ or $(ii)$ holds, then it follows from Theorem
  \ref{cmgperfuacimpliesgorenstein} and part (a) that $R$ is
  Gorenstein of Krull dimension $2$. On the other hand, if $R$ is such
  a ring, then there exists by Lemma \ref{koszul-p} a perfect
  $R$-complex of $\cat{G}$-level $3$, and $3$ is by \Cref{glevelbound}
  an upper bound on the $\cat{G}$-level of complexes in $\catdbf$.

  Now, let $d > 2$. If the equality in $(i)$ or $(ii)$ holds, then it
  follows from Theorem \ref{cmgperfuacimpliesgorenstein} that $R$ is
  Gorenstein of Krull dimension at most $d$, and by the induction
  hypothesis the Krull dimension is at least $d$. On the other hand,
  if $R$ is such a ring, then there exists by Lemma \ref{koszul-p} a
  perfect $R$-complex of $\cat{G}$-level $d+1$, and $d+1$ is by
  \Cref{glevelbound} an upper bound on the $\cat{G}$-level of
  complexes in $\catdbf$.
\end{proof}

\begin{example} \label{ex_lower_bound} Let $R$ be local with maximal
  ideal $\mfm$, and let $M$ be a complex in $\catdbf$ or a derived
  $\mathfrak{m}$-complete $R$-complex. Assume that $M$ has finite
  Gorenstein dimension and set $s\coloneqq \sup \h[*]{M}$.  If
  $\Gamma_I \h[s]{M}\neq 0$ holds, then we have
  \begin{align*}
    \gdim M &=\depth R -\depth M     \\
            &\geq \depth R -\depth(I,M) -\dim R/I \\
            &=  \depth R +s-\dim R/I \:.
  \end{align*}
  The first equality is by \Cref{auslanderbridger}. The inequality
  holds by work of Iyengar \cite[Proposition 5.5(4)]{Iyengar:1999} if
  $M$ is in $\catdbf$ and by work of Christensen and Ferraro
  \cite[Corollary 1.3]{Christensen/Ferraro:2024} if $M$ is derived
  $\mathfrak{m}$-complete. The last equality is standard when
  $\Gamma_I \h[s]{M}\neq 0$, see \cite[Theorem
  14.3.16]{Christensen/Foxby/Holm:2024}.  Now, per \cite[Theorem
  3.3]{Awadalla/Marley:2022} one has:
  \[
    \level{\cat G}{M} \geq \gdim M -s+1.
  \]
  Hence, combining these two inequalities, one gets:
  \[
    \level{\cat G}{M} \geq \depth R -\dim R/I +1.
  \]
\end{example}
  
This example gives another proof that the $\cat{G}$-level of the
Koszul complex on a sequence of parameters in a Cohen--Macaulay local
ring $R$ is exactly $\dim R+1$, cf.~\Cref{koszul-p}.

  \begin{example}
    Let $R$ be local and Cohen--Macaulay. If
    $\underline{x}\coloneqq x_1, \ldots, x_n$ is a generating set for
    a proper ideal $I$ of $R$, then \Cref{ex_lower_bound} yields
    \[
      \level{\cat G}{K(\underline{x};R) } \geq \dim R-\dim R/I+1.
    \]
    Further, if $\underline{x}$ forms a (partial) sequence of
    parameters for $R$, then
    \begin{equation*}
      \level{\cat G}{K(\underline{x}; R)}= \dim R -\dim R/I+1=n+1 
    \end{equation*}
    holds. Indeed, one has
    \[
      \level{\cat G}{K(\underline{x}; R)} \leq
      \level{R}{K(\underline{x}; R)} \leq n+1
    \]
    where the first equality holds as $R$ belongs to $\cat G$ and the
    second inequality is from the length of the Koszul complex, see
    \cite[Lemma 2.5.2]{Avramov/Buchweitz/Iyengar/Miller:2010}.
  \end{example}

  Finally, we remark that the $\cat{G}$- and $R$-levels of a Koszul
  complex can differ.

\begin{example}
  Let $R$ be local and Gorenstein of positive Krull dimension. Let $M$
  be a complex in $\catdbf$ with $s \coloneqq \sup \h[*]{M}$. If
  $\Gamma_{\mathfrak{m}} \h[s]{M}\neq 0$, then we have
  \[
    \level{\cat G}{M}=\dim R+1 \:.
  \]
  One inequality holds by \Cref{glevelbound} and the opposite
  inequality comes from Example \ref{ex_lower_bound}. Thus, for the
  Koszul complex $K$ on a set of generators of the maximal ideal of
  $R$ one has $\level{\cat G}{K}=\dim R+1$ while
  $\level{R}{M}=\operatorname{edim}R+1$ holds by \cite[Theorem
  4.2(1)]{Altmann/Grifo/Montano/Sanders/Vu:2017}. Thus if $R$ is not
  regular, then one has $\level{\cat G}{K} < \level{R}{K}$.
\end{example}


\section{Gorenstein injective levels}

\noindent
The next construction is dual to the Adams resolution in the sense of
J.D.~Christensen \cite{Christensen:1998}.

\begin{construction}
  \label{adams}
  Let $M$ be an $R$-complex. There is a graded injective $R$-module
  $I$ and an injective homomorphism
  $\bar{\iota}\colon \h{M} \hookrightarrow I$ of graded
  $R$-modules. By the graded-injectivity of $I$, this map lifts to a
  graded homomorphism $\co{M} \to I$. Considering $I$ as a complex
  with zero differential, this yields a morphism
  $\iota \colon M \to I$ of complexes.

  Set $\mho^0_R(M) \coloneqq M$ and
  $\mho^1_R(M) \coloneqq \cone(\iota)$. Applying the construction
  above recursively, set
  $\mho^{n+1}_R(M) \coloneqq \mho_R^1(\mho_R^n(M))$ for $n \geq
  1$. Taken together the ensuing triangles
  \begin{equation*}
    \mho^n(M) \xrightarrow{\iota^n} I^n \longrightarrow \mho^{n+1}(M) \longrightarrow
  \end{equation*}
  form an injective Adams resolution for $M$. They induce short exact
  sequences in homology
  \begin{equation*}
    0 \longrightarrow \h{\mho_R^n(M)} \longrightarrow I^n \longrightarrow \h{\mho^{n+1}(M)} \longrightarrow 0
  \end{equation*}
  for every $n \ge 0$.
\end{construction}

\begin{lemma}
  \label{SyzygyInequality}
  Let $\class{C}$ be a collection of objects in $\cat{D}$ that
  contains all injective $R$-modules. For every complex $M$ in
  $\catdb$ and every $n \ge 0$ there is an inequality
  \begin{equation*}
    \level{\class{C}} M \leq \level{\class{C}}\mho^n(M) + n 
  \end{equation*}
\end{lemma}

\begin{proof}
  It suffices to show that one has
  \begin{equation*}
    \level{\class{C}} M \leq \level{\class{C}}\mho^1(M) + 1 \:.
  \end{equation*}
  To this end consider the triangle
  \begin{equation*}
    M \overset{\iota^0}{\longrightarrow} I^0 \longrightarrow \mho^1(M) \longrightarrow
  \end{equation*}
  and notice that $I^0$ is a bounded complex of injective modules with
  zero differential. Now \cite[Lemma
  2.4(2)]{Avramov/Buchweitz/Iyengar/Miller:2010} yields the desired
  inequality.
\end{proof}

\begin{lemma}
  \label{Splice}
  Let $M \in \cat{D}$. For every integer $n \geq 1$ there is an exact
  sequence
  \begin{equation*}
    0 \longrightarrow \h{M}^\oplus \longrightarrow E^0 \longrightarrow E^1
    \longrightarrow  \cdots \longrightarrow E^{n-1}
    \longrightarrow \h{\mho^n(M)}^\oplus \longrightarrow 0
  \end{equation*}
  where each $E^i$ is an injective $R$-module.
\end{lemma}

\begin{proof}
  The claim holds since for every $i \ge 0$ there is an exact sequence
  of $R$-modules
  \begin{equation*}
    0 \longrightarrow \h{\mho^n(M)}^\oplus \longrightarrow (I^n)^\oplus
    \longrightarrow \h{\mho^{n+1}(M)}^\oplus \longrightarrow 0
  \end{equation*}
  from \Cref{adams}.
\end{proof}

\begin{chunk}
  \label{subquotientSES}
  In the next several results, we make use of the following exact
  sequences from \cite[Proposition
  2.2.12]{Christensen/Foxby/Holm:2024}:
  \begin{gather}
    \label{newtwoequation}
    0 \longrightarrow \cy[i]{M} \longrightarrow M_i \longrightarrow \bo[i-1]{M} \longrightarrow 0\\
    \label{newthreeequation}
    0 \longrightarrow \bo[i]{M} \longrightarrow M_i \longrightarrow \co[i]{M} \longrightarrow 0\\
    \label{newoneequation}
    0 \longrightarrow \h[i]{M} \longrightarrow \co[i]{M} \longrightarrow \bo[i-1]{M} \longrightarrow 0 \\
    \label{newzeroequation}
    0 \longrightarrow \bo[i]{M} \longrightarrow \cy[i]{M}
    \longrightarrow \h[i]{M} \longrightarrow 0
  \end{gather}
\end{chunk}

\begin{lemma}
  \label{finitegdimofZandB}
  Suppose $M$ is a left bounded $R$-complex such that $M_i$ and
  $\h[i]{M}$ have finite Gorenstein injective dimension for all
  $i$. Then $\bo[i]{M}$, $\cy[i]{M}$, and $\co[i]{M}$ have finite
  Gorenstein injective dimension for all $i$ as well.
\end{lemma}

\begin{proof}
  If $\cy[i]{M}$ has finite Gorenstein injective dimension, then the
  modules $\bo[i-1]{M}$ and $\co[i]{M}$ have finite Gorenstein
  injective dimension by \ref{subquotientSES}(\ref{newtwoequation})
  and \ref{subquotientSES}(\ref{newoneequation}), and then
  $\cy[i-1]{M}$ has finite Gorenstein injective dimension by
  \ref{subquotientSES}(\ref{newzeroequation}). The assertion thus
  follows by descending induction, as $\cy[i]{M} = 0$ for $i \gg 0$.
\end{proof}

\begin{proposition}
  \label{xfinitegdim}
  Let $M$ be a complex in $\catdb$. If the module $\h{M}^\oplus$ has
  finite Gorenstein injective dimension, then $M$ has finite
  Gorenstein injective dimension.
\end{proposition}

\begin{proof}
  Let $I$ be a semi-injective resolution of $M$.  Since every module
  $I_i$ has finite Gorenstein injective dimension, it follows from
  Lemma \ref{finitegdimofZandB} that $\cy[i]{I}$ has finite Gorenstein
  injective dimension for all $i$. Since $\h[i]{I}=0$ holds for
  $i \ll 0$, one has $\Gid{M} \le \Gid{\cy[i]{M}} - i$ for $i \ll 0$.
\end{proof}

In the following $\cat{GI}$ denotes the collection of Gorenstein
injective $R$-modules.

\begin{lemma}
  \label{finitegidimofBandC}
  Let $M$ be a left bounded complex of modules from $\cat{GI}$. If one
  has $\Gid{\h[i]{M}} \le 1$ for all $i$, then the modules $\bo[i]{M}$
  and $\co[i]{M}$ belong to $\cat{GI}$ for all $i$.
\end{lemma}

\begin{proof}
  If $\bo[i]{M}$ is in $\cat{GI}$, then $\co[i]{M}$ is in $\cat{GI}$
  by \ref{subquotientSES}(\ref{newthreeequation}), and per the
  assumption on $\h[i]{M}$ it follows from
  \ref{subquotientSES}(\ref{newzeroequation}) that $\bo[i-1]{M}$ is in
  $\cat{GI}$. The assertion thus follows by descending induction as
  one has $\bo[i]{M}=0$ for $i \gg 0$.
\end{proof}

The next theorem is dual to \Cref{mainglevelthm}; it does not follow
from work in \cite{AAITY-14} which only deals with resolving
categories.

\begin{proposition}
  \label{GIInequality}
  For every complex $M$ in $\catdb$ one has:
  \begin{equation*}
    \level{\cat{GI}}{M} \leq \max\{2, \Gid\h{M}^\oplus + 1\}
  \end{equation*}
\end{proposition}

\begin{proof}
  We may assume that $M$ is nonzero and that the Gorenstein injective
  dimension of $\h{M}^\oplus$ is finite; otherwise there is nothing to
  prove. Now the Gorenstein injective dimension of $M$ is finite by
  \Cref{xfinitegdim}. Thus, we may assume that $M$ is a bounded
  complex of Gorenstein injective modules.

  First, notice that if $\bo{M}^\oplus$ and $\co{M}^\oplus$ belong to
  $\cat{GI}$, then \cite[Lemma
  2.4(2)]{Avramov/Buchweitz/Iyengar/Miller:2010} applied to the exact
  sequence
  \[0 \longrightarrow \bo{M} \longrightarrow M\longrightarrow
    \Sigma\co{M} \longrightarrow 0 \] from
  \ref{subquotientSES}(\ref{newthreeequation}) yields
  \begin{align*}
    \level{\cat{GI}}{M} & \leq \level{\cat{GI}}{\bo{M}} +
                          \level{\cat{GI}}{\co{M}} \\
                        & = \level{\cat{GI}}{\bo{M}^\oplus} +    \level{\cat{GI}}{\co{M}^\oplus} \\
                        & \le 2\,,
  \end{align*}
  where the equality holds as the complexes $\bo{M}$ and $\co{M}$ have
  zero differentials. In particular, the asserted inequality holds by
  \Cref{finitegidimofBandC} in case $\Gid{\h{M}}^{\oplus} \le 1$.

  For $n \coloneqq \Gid{\h{M}^\oplus} > 1$ we see that
  $\Gid{\h{\mho^{n-1}_R(M)}^\oplus}=1$ holds by \Cref{Splice}, so the
  modules $\bo{\mho^{n-1}_R(M)}^{\oplus}$ and
  $\co{\mho^{n-1}_R(M)}^{\oplus}$ belong to $\cat{GI}$ by
  \Cref{finitegidimofBandC}. By the argument above, one now has
  $\level{\cat{GI}}{\mho^{n-1}_R(M)} \le 2$, so
  \Cref{SyzygyInequality} yields $\level{\cat{GI}}{M} \le n+1$.
\end{proof}

\begin{remark}
  \label{GIattain}
  The upper bound on $\cat{GI}$-levels provided in \Cref{GIInequality}
  can be attained: Indeed, for an $R$-module $M$ one has
  $\level{\cat{GI}}{M} = \Gid{M}+1$ by \cite[Lemma
  2.5.2]{Avramov/Buchweitz/Iyengar/Miller:2010} and work of Mifune
  \cite[Corollary 4.3]{LB2}. Thus the bound is attained by every
  $R$-module of positive Gorenstein injective dimension, and one can
  thus rewrite the bound as
  \begin{equation*}
    \level{\cat{GI}}{M} \leq \max\{2, \level{\cat{GI}}{\h{M}^\oplus}\} \,.
  \end{equation*}
  Next we show that a complex $M$ with $\h{M}^\oplus$ in $\cat{GI}$
  may have $\cat{GI}$-level $2$.
\end{remark}

\begin{example}
  \label{GIexa}
  Let $k$ be a field and consider the local ring
  $R \coloneqq k[x]/(x^2)$. Let $K$ be the Koszul complex on $x$; as
  $R$ is artinian and Gorenstein, $\h{K}^\oplus$ is a Gorenstein
  injective $R$-module.  By \Cref{GIInequality} one has
  $\level{\cat{GI}}{K} \le 2$. If the $\cat{GI}$-level of $K$ were
  $1$, then $K$ would be isomorphic to $\h{K}$ in the derived
  category, which as shown in \cite[Example
  3.10]{Awadalla/Marley:2022} is absurd. Thus
  $\level{\cat{GI}}{K} = 2$ holds.
\end{example}

The following is an immediate consequence of \Cref{GIInequality}:

\begin{theorem}
  \label{glevelbound-I}
  Let $R$ be Gorenstein and $M$ a complex in $\catdb$. One has
  \[\level{\cat{GI}}{M} \le \max\{2, \dim{R}+1\}\,.\]
\end{theorem}

As another application of \Cref{GIInequality} we prove a formula,
reminiscent of the Bass Formula, for Gorenstein injective level. The
assumption that $R$ has positive depth is necessary, for $R$ and $K$
as in \Cref{GIexa} one has $\level{\cat{GI}}{K} = 2$.

\begin{theorem}
  \label{GIBass}
  Let $R$ be a commutative noetherian local ring of positive depth and
  $M$ a complex in $\catdbf$. If $\Gid \h{M}^\oplus$ is finite, then
  one has
  \begin{equation*}
    \level{\cat{GI}}{M} = \depth R + 1
  \end{equation*}
\end{theorem}

\begin{proof}
  By Theorem~\ref{GIInequality}, we have
  $\level{\cat{GI}}{M} \leq \max\{2, \Gid \h{M}^\oplus + 1\}$. Since
  $\h{M}^\oplus$ is a finitely generated $R$-module, the Bass Formula
  for Gorenstein injective dimension, see \cite[Corollary
  19.2.14]{Christensen/Foxby/Holm:2024}, yields
  $\Gid \h{M}^\oplus = \depth R$. By assumption the depth of $R$ is
  positive, so one now has $\level{\cat{GI}}{M} \leq \depth R + 1$.

  For the opposite inequality, note first that $\Gid{M}$ is finite by
  \Cref{xfinitegdim}. Per \cite[Corollary 4.3]{LB2} one has
  \begin{equation*}
    \level{\cat{GI}}{M} \geq \Gid M + \inf M + 1 \:,
  \end{equation*}
  and $\Gid M + \inf M = \depth R$ holds by \cite[Theorem
  19.2.40]{Christensen/Foxby/Holm:2024}.
\end{proof}

We close out this section by remarking that that there is a Bass
formula for levels with respect to the class of injective $R$-modules.
Indeed, the proof of \cite[Theorem
5.5]{Avramov/Buchweitz/Iyengar/Miller:2010} dualizes to give an upper
bound on $\cat{I}$-levels, where $\cat{I}$ denotes the class of
injective $R$-modules.

\begin{proposition}
  \label{IInequality}
  For every complex $M$ in $\catdb$ one has:
  \begin{equation*}
    \level{\cat{I}}{M} \leq \idim{\h{M}^\oplus} + 1 \:.
  \end{equation*}
\end{proposition}

\begin{theorem}
  \label{Bass}
  Let $R$ be a noetherian local ring and $M$ a complex in $\catdbf$.
  If $\idim{\h{M}^\oplus}$ is finite, then one has
  \begin{equation*}
    \level{\cat{I}}{M} = \depth R+1 \:.
  \end{equation*}
\end{theorem}

\begin{proof}
  One has $\level{\cat{I}}{M} \le \depth R+1$ by \Cref{IInequality}
  and the classic Bass formula; the opposite inequality holds by
  \cite[Corollary 4.3]{LB2}.
\end{proof}


\section{Flat and Gorenstein flat levels}

Like \Cref{mainglevelthm} the next result follows from \cite[Theorem
4.1]{AAITY-14}.

\begin{proposition}
  \label{CInequality}
  Let $\class{C}$ be the class of flat $R$-modules, Gorenstein
  projective $R$-modules, or Gorenstein flat $R$-modules. For every
  complex $M$ in $\catdb$ one has
  \begin{equation*}
    \level{\class{C}}M \leq \max\{2, \class{C}\textnormal{-}\dim_R \h{M}^\oplus + 1\}
  \end{equation*}
  where $\class{C}$-$\dim_R$ means flat dimension, Gorenstein
  projective dimension, or Gorenstein flat dimension, respectively.
\end{proposition}

Like the bounds in Propositions~\ref{mainglevelthm} and
\ref{GIInequality} the ones provided above can be attained. First
consider the case where $\class{C}$ is the class $\cat{GP}$ of
Gorenstein projective $R$-modules. For every $R$-module $M$ one has
$\level{\cat{GP}}{M} = \GPdim{M} + 1$ by \cite[Corollary
3.4]{Awadalla/Marley:2022}. Further, for $R$ and $K$ as in
\Cref{GIexa}, the argument in \cite[Example
3.10]{Awadalla/Marley:2022} shows that while $\h{K}^\oplus$ is in
$\cat{GP}$ one has $\level{\cat{GP}}{K} = 2$.

\begin{lemma}
  \label{Ceq}
  Let $\class{C}$ be the class of flat $R$-modules or Gorenstein flat
  $R$-modules. For every $R$-module $M$ one has
  \begin{equation*}
    \level{\class{C}}M = \class{C}\textnormal{-}\dim_R M + 1 \:,
  \end{equation*}
  where $\class{C}$-$\dim_R$ means flat dimension dimension or
  Gorenstein flat dimension.
\end{lemma}

\begin{proof}
  First we consider the case where $\class{C}$ is the class $\cat{F}$
  of flat $R$-modules. Note that the inequality
  $\level{\cat{F}}{M}\leq \fdim[R]{M} +1$ holds by \cite[Lemma
  2.5.2]{Avramov/Buchweitz/Iyengar/Miller:2010}. For the other
  inequality, let $E$ be a faithfully injective $R$-module. By
  \cite[Lemma 2.4(6)]{Avramov/Buchweitz/Iyengar/Miller:2010} there is
  an inequality,
  \begin{equation*}
    \level{\Hom{\cat{F}}{E}}\Hom{M}{E} \leq \level{\cat{F}}{M} \:.
  \end{equation*}
  For every flat $R$-module $F$, the module $\Hom{F}{E}$ is injective
  by flat--injective duality. Thus one has
  \begin{equation*}
    \level{\cat{I}}{\Hom{M}{E}} \leq \level{\Hom{\cat{F}}{E}}{\Hom{M}{E}} \:.
  \end{equation*}
  As $\level{\cat{I}}{\Hom{M}{E}} = \idim{\Hom{M}{E}} + 1$ holds by
  \Cref{GIattain} and one has $\idim\Hom{M}{E} = \fdim{M}$ by
  flat-injective duality, the asserted inequality follows.

  The same argument applies when $\class{C}$ is the class of
  Gorenstein flat $R$-modules; only one refers to \cite[Proposition
  9.3.20]{Christensen/Foxby/Holm:2024} in place of flat--injective
  duality.
\end{proof}

\begin{remark}
  In the flat and Gorenstein flat cases, the upper bound established
  in \Cref{CInequality} can be attained, as
  $\level{\class{C}}M = \class{C}\textnormal{-}\dim_RM + 1$ holds for
  every $R$-module $M$ by \Cref{Ceq} .

  Next we construct a complex $M$ with
  $\level{\cat{F}}{M} = 2 = \level{\cat{GF}}{M}$ though the module
  $\h{M}^\oplus$ is flat and, hence, Gorenstein flat.
\end{remark}

\begin{example}
  Set $R \coloneqq \mathbb{R}[x, y, z]$ and let $Q$ be the field of
  fractions of $R$. As $R$ is regular, the projective dimension of an
  $R$-module is at most $3$, and every Gorenstein flat $R$-module is
  flat; in symbols $\cat{GF} = \cat{F}$. By a result of Osofsky
  \cite{BLO68}, the projective dimension of the flat $R$-module $Q$ is
  at least two.  Let $P$ be a free resolution of $Q$, and consider the
  hard truncation of $P$ at degree 1, denoted $P_{\leq 1}$. One has
  $\h[0]{P_{\leq 1}} \cong Q$ and
  $\h[1]{P_{\leq 1}} \cong \Omega_R^2(Q)$, the second syzygy of $Q$
  which is also a flat $R$-module. Thus, $\h{P_{\leq 1}}^\oplus$ is
  flat, whence \Cref{CInequality} yields
  $\level{\cat{F}}P_{\leq 1} \leq 2$; we claim that equality holds.
  Assume towards a contradiction that $\level{\cat{F}} P_{\leq 1} = 1$
  holds. In the derived category, $P_{\leq 1}$ is now isomorphic to
  its homology. As $P_{\leq 1}$ is semi-projective, there is a
  quasi-isomorphism $\pi$ from $P_{\leq 1}$ to its homology:
  \begin{center}
    \begin{tikzcd}
      0 \arrow[r] & P_1 \arrow[r]\arrow[d, "\pi_1"] & P_0 \arrow[r]\arrow[d, "\pi_0"] & 0 \\
      0 \arrow[r] & \Omega^2_R(Q) \arrow[r, "0"] & Q \arrow[r] & 0
    \end{tikzcd}
  \end{center}
  As $\pi$ is a quasi-isomorphism, its mapping cone is acyclic. Thus,
  the sequence
  \begin{equation*}
    0 \longrightarrow P_1 \longrightarrow P_0 \oplus \Omega^2_R(Q) \longrightarrow Q \longrightarrow 0
  \end{equation*}
  is exact.  One has $\pdim{P_0 \oplus \Omega_R^2(Q)} = \pdim{Q} - 2$
  as $\Omega_R^2(Q)$ is a second syzygy of $Q$. Since the projective
  dimension of $Q$ is at most 3, the projective dimension of
  $P_0 \oplus \Omega_R^2(Q)$ is at most 1. However, from the exact
  sequence above one gets
  \begin{align*}
    \pdim{Q} 	&\leq \max\{\pdim{P_1}+1, \pdim{P_0 \oplus \Omega_R^2(Q)}\} \\
		&\leq \max\{1, 1\} \\
		&\leq 1
  \end{align*}
  by \cite[Corollary 8.1.9]{Christensen/Foxby/Holm:2024}. This is a
  contradiction, and so $\level{\cat{F}} Q = 2$.
\end{example}

\section*{Acknowledgments}

\noindent
We thank Srikanth Iyengar, Janina Letz, Tom Marley, Josh Pollitz and
Ryan Watson for helpful conversations. We also thank Ryo Takahashi for
alerting us that \Cref{mainglevelthm}, for which we had originally
included a proof, follows from \cite[Theorem 4.1]{AAITY-14}.


\providecommand{\bysame}{\leavevmode\hbox to3em{\hrulefill}\thinspace}
\providecommand{\MR}{\relax\ifhmode\unskip\space\fi MR }
\providecommand{\MRhref}[2]{%
  \href{http://www.ams.org/mathscinet-getitem?mr=#1}{#2} }
\providecommand{\href}[2]{#2}

\end{document}